\documentclass[a4paper]{amsart}
\usepackage{soul} %for strikethough
\usepackage[all]{xy}
\usepackage{amsmath}
\usepackage{amstext}
\usepackage{amssymb}
\usepackage{amscd}
\usepackage{soul}
\usepackage{enumerate}
\usepackage{graphicx}
\usepackage{tikz}
\usepackage{mathrsfs}
\usepackage{rotating}
\usepackage{tikz-cd}
\usepackage{color}
\usepackage{mathtools}
\xyoption{all}
\usepackage{pstricks}
\usepackage{graphicx}
\usepackage[bookmarksnumbered,colorlinks,plainpages]{hyperref}
\hypersetup{
	colorlinks=true,
	citecolor=blue
}
\newtheorem{theorem}{Theorem}[section]

\newtheorem{lemma}[theorem]{Lemma}
\newtheorem{proposition}[theorem]{Proposition}
\newtheorem{definition-proposition}[theorem]{Definition-Proposition}
\newtheorem{problem}[theorem]{Problem}

\theoremstyle{definition}
\newtheorem{definition}[theorem]{Definition}

\newtheorem{example}[theorem]{Example}

\numberwithin{equation}{section}

\newcommand{\CC}{\mathcal{C}}
\newcommand{\OO}{{\mathcal O}}
\newcommand{\UU}{\mathcal{U}}

\renewcommand{\AA}{\mathcal{A}}
\renewcommand{\L}{\mathbb{L}}
\newcommand{\X}{\mathbb{X}}

\newcommand{\Z}{\mathbb{Z}}
\newcommand{\mdots}{\begin{turn}{87}$\ddots$ \end{turn}}

\renewcommand{\c}{\vec c}
\newcommand{\vdelta}{\vec{\delta}}
\newcommand{\y}{\vec y}
\newcommand{\x}{\vec x}

\newcommand{\vell}{\vec {\ell}}
\newcommand{\z}{\vec z}

\newcommand{\s}{\vec s}

\newcommand{\w}{\vec{\omega}}

\newcommand{\cut}{\ar@{-}@[|(5)]}

\newcommand{\Hom}{\operatorname{Hom}\nolimits}

\newcommand{\End}{\operatorname{End}\nolimits}
\newcommand{\Ext}{\operatorname{Ext}\nolimits}

\newcommand{\bo}{\operatorname{b}\nolimits}

\newcommand{\RHom}{\mathbf{R}\strut\kern-.2em\operatorname{Hom}\nolimits}

\DeclareMathOperator{\add}{\mathsf{add}}

\DeclareMathOperator{\thick}{\mathsf{thick}}
\DeclareMathOperator{\CM}{\mathsf{CM}}

\DeclareMathOperator{\moduleCategory}{\mathsf{mod}} \renewcommand{\mod}{\moduleCategory}
\DeclareMathOperator{\coh}{\mathsf{coh}}
\newcommand{\DDD}{\mathsf{D}}
\newcommand{\KKK}{\mathsf{K}}

\DeclareMathOperator{\vect}{\mathsf{vect}}
\DeclareMathOperator{\proj}{\mathsf{proj}}

\tikzset{every picture/.style={line width=0.75pt}} %set default line width to 0.75pt   

\begin{document}

	\title[An open question for $d$-tilting bundles on GL projective spaces]{An open question for $d$-tilting bundles on Geigle-Lenzing projective spaces}

\author[J. Chen and W. Weng] {Jianmin Chen and Weikang Weng$^*$}

\thanks{$^*$ the corresponding author}
\makeatletter \@namedef{subjclassname@2020}{\textup{2020} Mathematics Subject Classification} \makeatother

	
	
	\begin{abstract} We construct a family of $2$-tilting bundles on a Geigle-Lenzing projective space of type $(2,2,p,q)$ via the action of iterated $2$-APR mutations. As an application, we give some non-examples for an open question raised by  Herschend, Iyama, Minamoto and Oppermann in the paper ``Representation theory of Geigle-Lenzing complete intersections".

	\end{abstract}
	
	\maketitle
	\section{Introduction}

	Let $k$ be an algebraically closed  field and  integers $d,n\ge 1$.  Following \cite{HIMO}, the \emph{Geigle-Lenzing}  (\emph{GL}) \emph{complete intersection} is  the $k$-algebra
	$$R:=k[T_0, \ldots, T_d,X_1,\dots,X_n] / (X_i^{p_i} -\ell_i \mid 1\le i \le n),$$
	where $\ell_1, \dots, \ell_n$ are linear forms in $k[T_0, \ldots, T_d]$ as linear independent as possible, and $p_1,\dots,p_n\ge 2$ are  integers  called \emph{weights}. It is graded by an abelian group  $$\L:= \langle \x_1,\ldots,\x_n, \c \rangle / \langle p_i \x_i - \c \mid 1\le i \le n\rangle.$$  	
	The category $\coh \X$ of coherent sheaves on  \emph{Geigle-Lenzing} (\emph{GL}) \emph{projective space} $\X$ of dimension $d$ is defined by applying Serre's construction \cite{S1} to the $\L$-graded $k$-algebra $R$. The GL projective spaces are higher dimensional analogs of weighted projective lines in the sense of Geigle-Lenzing \cite{GL}. They link various subjects, including  representation theory, singularity theory, homological mirror symmetry and (non-commutative) algebraic geometry, see e.g. \cite{CRW, FU,HIMO,KST1, KST2, KLM}.
	
	Denote by $\CM^{\L}R$ the category  of $\L$-graded Cohen-Macaulay  modules, by $\underline{\CM}^{\L}R$ its associated stable category, and by $\DDD^{\bo}(\coh \X)$ the bounded derived category of $\coh \X$.
	 In \cite{HIMO},  Herschend, Iyama, Minamoto and Oppermann studied when $\underline{\CM}^{\L}R$ 
	 or $\DDD^{\bo}(\coh \X)$ has a $d$-tilting object.  
	 %(that is, a tilting object whose endomorphism algebra has global dimension at most $d$). 
	 This property holds significance within the framework of higher Auslander-Reiten theory \cite{HIO, I1}. In fact,  the endomorphism algebra of a $d$-tilting bundle  on $\X$ is a $d$-representation infinite algebra in sense of \cite{HIO}. They also investigated that under certain conditions,  a $d$-tilting object in $\underline{\CM}^{\L}R$ can lift to a $d$-tilting bundle on $\X$. Naturally, they posed the following problem.
	
	\begin{problem}$($\cite[Problem 7.39]{HIMO}$)$ \label{problem}
		\label{d-tilting bundle is d-tilting object}
		Let $V\in\CM^{\L}R$ be a $d$-tilting bundle on $\X$. Is $V$ a $d$-tilting object in $\underline{\CM}^{\L}R$?
	\end{problem}
	
	In the case $d=1$ or $n\le d+1$,  this  is true, see \cite{HIMO,KLM}. For the general case, a partial answer to Problem \ref{problem} was given in \cite{HIMO}, see Theorem \ref{partical ans}.

	 Based on the observation that the action of $d$-APR mutations preserves $d$-tilting bundles, see Lemma \ref{mutation again}.  In the case $d=2$ and $n=4$, we construct a family of $2$-tilting bundles on a GL projective space of  dimension $2$, and thus obtain a family of $2$-representation infinite algebras. 
	As an application, we find some  non-examples of the Problem \ref{problem}. Furthermore, we also show that the converse of Theorem \ref{partical ans} is not true in general, see Example \ref{example 2}.

	 Let $\s:=\sum_{i=1}^{4}\x_i$, $\w:=\c-\s$ and $\vdelta :=2\c+2\w$. For any element $\vell=\sum_{i=1}^{4}\ell_i\x_i$ in the internal $[\s,\s+\vdelta]$, let	$U^{\vell}:=\rho(R/(X_i^{\ell_i}\mid 1\le i\le 4)),$ where $\rho$ denotes the composition $\DDD^{\bo}(\mod^{\L}R)\to\DDD_{\rm sg}^{\L}(R)\xrightarrow{\sim} \underline{\CM}^{\L}R$ and the set $\mathcal{S}$  consists of  elements $\x=\sum_{i=1}^{4}\lambda_i\x_i+\lambda\c$ (in normal form) in $\L$ 
	satisfying 
	\begin{align} \label{S=}
		0 < 2 \lambda + \# \{1\le i \le 4\mid\lambda_i\ge1\} \le \#\{3\le i \le 4\mid\lambda_i \in \{0,1\}\}.
	\end{align}

	Our main theorem in this paper is the following.
	 	\begin{theorem} \label{main theorem} {\rm (See Theorem~\ref{main theorem_1}, Propositions~\ref{pro1}, \ref{pro2} for details)} 
		 		Let $d=2$, $n=4$ and $(R,\L)$ be a GL complete intersection of  type $(2,2,p,q)$ with integers $p,q \ge 2$, and $\X$ the corresponding GL projective space. Let $J$ be an  upset restricted to $[\s,\s+\vdelta]$ and $I:=[\s,\s+\vdelta]\setminus J$.
		 		We put
		 		\[M:=(\bigoplus_{\vell\in I} U^{\vell} ) \oplus  (\bigoplus_{\z\in J} U^{\z}(-\w) ) \oplus(\bigoplus_{\x\in \mathcal{S}}R(\x)).\]
		 		Then  $M$ gives a $2$-tilting bundle on $\X$. Moreover, we have the following.
		 		\begin{itemize}
			 			\item [(a)] If $p=2$, then $M$ is	not a $2$-tilting object  in $\underline{\CM}^{\L} R$ if and only if $J=[\s+m\x_4,\s+(q-2)\x_4]$ for some $1\le m \le q-2$.		
			 			\item [(b)] If $p=q=3$, then $M$ is	not a $2$-tilting object  in $\underline{\CM}^{\L} R$ if and only if $J=[\s+\x_i,\s+\x_3+\x_4]$ for $i=3,4$.
			 \end{itemize}
		 \end{theorem}	
			
	Recall from \cite{HIMO} that
	we have the following diagram which shows connections between the relevant notions (here the down arrow is a restriction).

		\[
	\begin{xy}
		(-33,0)*{\left\{\begin{array}{c}\mbox{$d$-tilting bundles on $\X$}\\ \mbox{contained in $\CM^{\L}R$}\end{array}\right\} \ }="A",
		(33,0)*{\ \left\{\begin{array}{c}\mbox{slices in $d$-cluster tilting}\\ \mbox{subcategories of $\CM^{\L}R$}\end{array}\right\}}="B",
		(-33,-20)*{\{\mbox{$d$-tilting objects in $\underline{\CM}^{\L}R$}\}  \ }="C",
		(33,-20)*{\ \left\{\begin{array}{c}\mbox{slices in $d$-cluster tilting}\\ \mbox{subcategories of $\underline{\CM}^{\L}R$}\end{array}\right\}}="D",
		
		\ar@{=} "A";"B"^{{\tiny \begin{array}{c} \mbox{\cite[Theorem 7.23]{HIMO}}   \end{array}}}
		%\ar@{=} "A";"B"^{{\tiny \begin{array}{c} \mbox{\cite[Theorem 7.23]{HIMO}}   \end{array}}}
		%\ar "A";"C"^{{\tiny \begin{array}{c} \mbox{(Theorem \ref{main theorem})}  \end{array}}}
		\ar "B";"D"_{{\tiny \begin{array}{c} \mbox{}  \end{array}}}
		\ar@{^{(}->} "C";"D"^{{\tiny \begin{array}{c} \mbox{\cite[Theorem 4.53]{HIMO}}   \end{array}}}
		%\ar@{^{(}->} "C";"D"^{{\tiny \begin{array}{c} \mbox{\cite[Theorem 4.53]{HIMO}}   \end{array}}}
	\end{xy}
	\]
		
	This diagram implies that the bottom inclusion is strict for the non-examples given in Theorem \ref{main theorem}. Furthermore,  the connectedness of $d$-tilting bundles in a $d$-cluster tilting subcategory under $d$-APR mutations is shown in Theorem \ref{connectedness}.

	The paper is organized as follows.  In Section \ref{sec: Preliminaries}, we collect some preliminaries about Geigle-Lenzing projective spaces, $d$-tilting bundles and $d$-APR mutations of them. In Section \ref{sec:main results}, we prove Theorem \ref{main theorem} via the action of iterated $d$-APR mutations. In Section \ref{sec:examples}, we give two  non-examples: one for the Problem \ref{problem}, and the other for the converse of  Theorem \ref{partical ans}.

	\section{Preliminaries} \label{sec: Preliminaries}
	In this section, we recall some basic notions and facts appearing in this paper.  We refer to \cite{HIMO, KLM} on Geigle-Lenzing complete intersections and projective spaces, and to \cite{I1,J1,JK} on higher dimensional Auslander-Reiten theory. Throughout this paper, we fix an algebraically closed  field $k$ and  integers $d,n\ge 1$.

		\subsection{$d$-cluster tilting subcategories} Let us start with recalling the notion of functorially finite subcategory of an additive category. Let $\AA$ be an additive category and $\CC$ a full subcategory of $\AA$.
	We say that $\CC$ is a \emph{contravariantly finite subcategory} of $\AA$ if for every object in $A\in\AA$, there exists a morphism
	$f:C\to A$ with $C \in \CC$ such that the induced morphism
	$$\Hom_{\AA}(C',C)\to\Hom_{\AA}(C',A)$$ is surjective for all $C'\in\CC$. Such a morphism $f$ is called a \emph{right $\CC$-approximation} of $A$. Dually we define  a \emph{left $\CC$-approximation} and a 
	\emph{covariantly finite subcategory}. We say that $\CC$ is \emph{functorially finite} if it is both contravariantly and covariantly finite.

	\begin{definition}(\cite[Definition 1.1]{I3}, \cite[definition 4.13]{J1}) A full subcategory $\CC$  of an exact category or triangulated category $\AA$ is called  \emph{$d$-cluster tilting subcategory} if it is a functorially finite subcategory of $\AA$ such that	
			\begin{align*}
				\CC&=\{X\in\AA\mid \Ext_{\AA}^i(\CC,X)=0 \text{ for any }  1\le i \le d-1\} \\
				&=\{X\in\AA\mid \Ext_{\AA}^i(X,\CC)=0\text{ for any }   1\le i \le d-1 \}
			\end{align*}
			and, in the exact case, $\CC$ is a both generating and cogenerating subcategory of $\AA$, that is, for any object $A\in \AA$, there exists an epimorphism $C \to A$ with $C \in \CC$ and a  monomorphism  $ A \to C'$ with  $C' \in \CC$. 
	\end{definition} 

	\subsection{Geigle-Lenzing complete intersections and projective spaces} Fix an $n$-tuple $(p_1,\dots,p_n)$ of all integers $p_i \ge 2$, called \emph{weights}. Consider the $k$-algebra
	$$R:=k[T_0, \ldots, T_d,X_1,\dots,X_n] / (X_i^{p_i} -\ell_i \mid 1\le i \le n),$$
	where $\ell_1, \dots, \ell_n$ are linear forms in $k[T_0, \ldots, T_d]$ in general position, that is, 
	any set of at most $d+1$ of the linear forms $\ell_i$ is linearly independent. 
	Let $\L$ be the abelian group on generators $\x_1,\ldots,\x_n,\c$ subject to the relations $$p_1 \x_1= \dots=p_n \x_n=:\c.$$ The element $\c$ is called the \emph{canonical element} of $\L$. Each element $\x$ in $\L$ can be uniquely written in a \emph{normal form} as 
	\begin{align}\label{equ:nor}
		\x=\sum_{i=1}^n\lambda_i\x_i+\lambda\c
	\end{align}
	with $ 0\le \lambda_i < p_i$ and $\lambda\in\Z$. We can regard $R$ as an $\L$-graded $k$-algebra by setting $\deg X_i:=\x_i$  and $\deg T_j:=\c$  for any $i$ and $j$, hence $R=\bigoplus_{\x\in \L } R_{\x}$, where $R_{\x}$ denotes the homogeneous component of degree $\vec{x}$. 
	The pair $(R,\L)$ is called \emph{Geigle-Lenzing} (\emph{GL})
	\emph{complete intersection} associated with $\ell_1,\ldots,\ell_n$ and $p_1,\ldots,p_n$. In particular, $R$ is regular if  $n\le d+1$; $R$ is a hypersurface if $d=n+2$.

	Let $\L_+$ be the submonoid of $\L$ generated by $\x_1,\ldots,\x_4,\c$. Then we equip $\L$ with the structure of a partially ordered set: $\x\le \y$ if and only if $\y-\x\in\L_+$. Each element $\x$ of $\L$ satisfies exactly one of the following two possibilities
	\begin{align}\label{two possibilities of x}
		\x\geq 0 \text{\quad or\quad}  \x\leq d\c+\vec{\omega},
	\end{align}
	where  $\w=\c-\sum_{i=1}^n \x_i$ is called the \emph{dualizing element} of $\L$. 
	Note that $R_{\x} \neq 0$ if and only if $\x \geq 0$ if and only if $\ell \geq 0$ in its normal form in (\ref{equ:nor}).  
	The interval in $\L$ is denoted by
	$[\x,\y] := \{\z \in\L \mid \x \le \z \le \y\}.$ Let $I \subset J \subset \L$. We say that $I$ is an \emph{upset} of $\L$ restricted to $J$ if $(I+\L_{+}) \cap J \subset I$ holds.
	
		We denote by $\CM^{\L} R$ the category of $\L$-graded (maximal) Cohen-Macaulay $R$-modules,  
	which is Frobenius, and thus its stable category $\underline{\CM}^{\L}R$ forms a triangulated category by a general result of Happel \cite{Hap1}. By results of Buchweitz \cite{Bu} and Orlov \cite{Orlov:2009}, there exists a triangle equivalence $\DDD^{\L}_{\rm sg}(R)\simeq\underline{\CM}^{\L}R$, where 
	$\DDD^{\L}_{\rm sg}(R):=\DDD^{\bo}(\mod^{\L}R)/\KKK^{\bo}(\proj^{\L}R)$
	is the \emph{singularity category} of $R$.
	The following presents a basic property in Cohen-Macaulay representation theory. Here, we denote by $D$ the $k$-dual, that is  $D(-):=\Hom_{k}(-,k)$.
	\begin{theorem}[\cite{HIMO}] \label{Auslander-Reiten-Serre duality} (Auslander-Reiten-Serre duality) There exists a functorial isomorphism for any $X,Y\in \underline{\CM}^{\L}R$:
		\begin{equation}\label{Auslander-Reiten-Serre duality CM}
			\Hom_{\underline{\CM}^{\L}R}(X,Y)\simeq D\Hom_{\underline{\CM}^{\L}R}(Y,X(\w)[d]).
		\end{equation}	
	\end{theorem}

	The category of \emph{coherent sheaves} on
	\emph{Geigle-Lenzing} (\emph{GL}) \emph{projective space} $\X$ of dimension $d$  is defined as the quotient category	
	\[ \coh \X :=\mod^{\L}R/\mod_0^{\L}R\]
	of the category of finitely generated $\L$-graded modules $\mod^{\L}R$ modulo its Serre subcategory $\mod^{\L}_0R$ consisting of finite dimensional modules. It follows from \cite{HIMO} that $\coh\X$ is a Noetherian abelian category of global dimension $d$. 
		
	Denote by $\pi:\mod^{\L}R\to\coh\X$ the natural functor, also called \emph{sheafification}. The object $\OO:=\pi(R)$ is called the \emph{structure sheaf} of $\X$, and $\L$ acts on $\coh \X$ by grading shift. Each line bundle has the form $\OO(\x)$ for a uniquely determined $\x$ in $\L$, and for any $\x,\y \in \L$ and $i\in\Z$, we have 
				\begin{equation}\label{extension spaces between line bundles}
						\Ext_{\X}^i(\OO(\x),\OO(\y))=\left\{
						\begin{array}{ll}
								R_{\y-\x}&\mbox{if $i=0$,}\\
								D(R_{\x-\y+\w})&\mbox{if $i=d$,}\\
								0&\mbox{otherwise.}
							\end{array}\right.
					\end{equation}	
										
	We denote by $\vect \X$ the full subcategory of $\coh \X$ formed by all vector bundles. Recall that the natural functor  $\pi:\mod^{\L}R\to\coh\X$ restricts to a fully faithful functor $\CM^{\L}R\to\vect\X$. Therefore $\CM^{\L}R$ has two exact structures, one is inherited from $\mod^{\L} R$, and the other is inherited from $\coh \X$. 
	The following result investigates a close relationship between them.
	
	\begin{proposition}\cite[Proposition 5.19(c)]{HIMO} \label{relationship}
			 For any integer $i$ with $0\le i\le d-1$, there exists a functorial isomorphism \[\Ext_{\mod^{\L}R}^i(X,Y)\simeq\Ext_{\X}^i(X,Y)\]
			for any $X\in\mod^{\L}R$ and $Y\in\CM^{\L}R$.
	\end{proposition}
 	
	In particular, we have $\Hom^{\L}_{R}(X,Y) \simeq \Hom_{\X}(X,Y)$ for any $X,Y\in\CM^{\L}R$.

	\subsection{$d$-tilting objects and slices}
	Let $\mathcal{T}$ be a triangulated category with a suspension functor $[1]$. In this paper, we only treat the case  $\mathcal{T}=\underline{\CM}^{\L} R$ or $\DDD^{\bo}(\coh \X)$.
	
	\begin{definition} An object $V \in \mathcal{T}$ is called \emph{tilting} in $\mathcal{T}$ if
		\begin{itemize}
			\item[(a)] $V$ is \emph{rigid}, that is $\Hom_{\mathcal{T}}(V,V[i])=0$ for all $i \neq 0$.
			\item[(b)] $\thick V=\mathcal{T}$, where $\thick V$ denotes by the smallest thick subcategory of $\mathcal{T}$ containing $V$.
		%	\item[(c)] $\End_{\mathcal{T}} (V)$ has global dimension at most $d$.		
		\end{itemize}
		If moreover $\End_{\mathcal{T}} (V)$ has global dimension at most $d$,  such a $V$ is called \emph{$d$-tilting}.
		 	A $d$-tilting object $V$ in  $\DDD^{\bo}(\coh \X)$ is called a $d$-\emph{tilting bundle} on $\X$ if $V \in \vect \X$. 
	\end{definition}

		\begin{definition}(c.f. \cite[Definition 7.16]{HIMO})\label{def.slice}  Let $\UU$ be a $d$-cluster tilting subcategory of $\CM^{\L} R$ (resp. $\underline{\CM}^{\L} R$). 
			An object $V \in \UU$ is called $\emph{slice}$  in $\UU$ if the following conditions are satisfied.
			\begin{itemize}
				\item[(a)] ${\UU}=\add\{ {V} (\ell\w)\mid\ell\in\Z\}$.
				\item[(b)] $\Hom_{\UU}({V},{V}(\ell\w))=0$ for any $\ell>0$.
			\end{itemize}
			In this case, $\add V$ meets each $\w$-orbit of indecomposable objects in $\UU$ exactly once.		
		\end{definition}
	 
	Recall from \cite[Theorem 4.53]{HIMO} that a $d$-tilting object $U$ in $\underline{\CM}^{\L} R$ induces a $d$-cluster tilting subcategory $\underline{\UU}:= \add \{U(\ell\w) \mid \ell \in \Z \}$ of $\underline{\CM}^{\L} R$, and further induces a  $d$-cluster tilting subcategory 
	$\UU:=\add \{ U(\ell\w),~R(\x) \mid \ell \in\Z,~\x\in \L \}$ of ${\CM}^{\L} R$. Moreover, $U$ is a slice in $\underline{\UU}$
	by \cite[Proposition 4.28]{HIMO} and \cite[Proposition 2.3(b)]{HIO}.

	The following theorem investigates the relationship between  $d$-tilting bundles and slices in $d$-cluster tilting subcategories, which will be used frequently later.
	
	\begin{theorem} {\rm (\cite[Theorem 7.23]{HIMO})} \label{tilting-cluster tilting}	
	For $V\in\CM^{\L}R$,  $V$ is a $d$-tilting bundle on $\X$  if and only if $V$ is a slice in a $d$-cluster tilting subcategories of $\CM^{\L}R$.
	\end{theorem}
	
	The following result gives a partial answer to Problem \ref{problem}.
	
	\begin{theorem}{\rm (\cite[Theorem 7.40(b)]{HIMO})} \label{partical ans}
		Let $V\in\CM^{\L}R$ be a $d$-tilting bundle on  $\X$. Take a decomposition $V=P\oplus W$, where $P\in\proj^{\L}R$ is a maximal projective direct summand of $V$. If $\underline{\End}^{\L}_R(W)=\End^{\L}_R(W)$, then $V$ is a $d$-tilting object in $\underline{\CM}^{\L}R$.
	\end{theorem}
	
	\subsection{$d$-APR mutations of $d$-tilting bundles} 
	In this section, we introduce the notion of $d$-APR (=Auslander–Platzeck–Reiten) mutation of $d$-tilting bundles, c.f. \cite[Definition 2.4]{MY} and \cite[Definition 4.12]{IO1}, which is a higher dimensional analog of APR-mutation in sense of \cite[Definitions 4.2]{G1}. 
  
		\begin{definition} \label{def.mutation}
		Let $V\in \CM^{\L}R$ be a $d$-tilting bundle on $\X$, and $V = V' \oplus V''$ be a direct sum decomposition such that $\Hom_{\X}(V'', V') = 0$. We set
		\begin{align*}
			\mu^+_{V'}(V)  := V'(-\w) \oplus V'' \ \text{ and } \
			\mu^-_{V''}(V)  := V' \oplus V''(\w).
		\end{align*}			
		If moreover $V'$ (resp. $V''$) is indecomposable, we call  $\mu^+_{V'}(V)$ (resp. $\mu^-_{V''}(V)$) a \emph{$d$-APR mutation}.
	\end{definition}

	\begin{lemma} \label{mutation again}
		In the setup of Definition~\rm\ref{def.mutation}, 
		$\mu^+_{V'}(V)$ and $\mu^-_{V''}(V)$ are both $d$-tilting bundles on $\X$.
	\end{lemma}
	
	\begin{proof}
		We only prove the case for $\mu^+_{V'}(V)$, as the other case can be shown by a similar argument. 		
		By Theorem \ref{tilting-cluster tilting}, $V$ is a slice in the $d$-cluster tilting subcategory $\UU:=\{V(\ell\w) \mid \ell \in \Z \}$ of $\CM^{\L} R$, and it is enough to show that  $\mu^+_{V'}(V)$ is also a slice in $\UU$.
		Clearly it satisfies condition (a) of Definition \ref{def.slice}. It remains to show  condition (b) there, that is,  $\Hom_{{\X}}({\mu^+_{V'}(V)},{\mu^+_{V'}(V)}(\ell\w))=0$ for any $\ell>0$. 
		This follows from the definition of slice $V$ and  $\Hom_{\X}(V'', V') = 0$.
	\end{proof}
	
	By a parallel argument as in \cite[Lemma 4.14]{IO1}, we have the following result.
	
	\begin{theorem}	\label{connectedness}
		Let $V\in \CM^{\L} R$ be a $d$-tilting bundle on $\X$ and $\UU:=\add \{V(\ell\w) \mid \ell \in \Z \}$ be the $d$-cluster tilting subcategory of $\CM^{\L} R$.  If moreover the quiver of $V$ contains no oriented cycles, then the action of iterated $d$-APR mutations of the set of  $d$-tilting bundles on $\X$ contained in $\UU$ is transitive.
	\end{theorem}
	\begin{proof} 
		Take a decomposition $V= \bigoplus_{i\in I} U_i$ into indecomposable modules in ${\CM}^{\L} R$. By Theorem \ref{tilting-cluster tilting},   let $V_1:=\bigoplus_{i\in I} U_i(a_i\w)$ and $V_2:=\bigoplus_{i\in I} U_i(b_i\w)$ be any two $d$-tilting bundles on $\X$ contained in $\UU$, where $a_i,b_i\in \Z$.
			Note that the quiver of $V$ contains no oriented cycles, then also the quiver of $\UU$. Thus we can 	number the indecomposable direct summands of $V_1$ as $V_1=W_1\oplus\dots \oplus W_n$ such that $\Hom(W_{i},W_{j})=0$ for any $i>j$. Therefore we may mutate and obtain $$\mu^-_{W_1}\circ\dots\circ\mu^-_{W_n} (V_1) =\mu^-_{V_1}(V_1)=V_1(\w).$$
		Hence without loss of general, we can  assume $a_i \ge b_i$ for all $i\in I$.   We set $J:=\{i\in I \mid a_i-b_i \text{ is maximal}\}$ and take a decomposition of $V_1$ into the direct sum of
		$$V_1'=\bigoplus_{j\in J} U_j(a_j\w) \ \text{ and } \ V_1''=\bigoplus_{t\in I\setminus J} U_t(a_t\w).$$ 
		Then for any $j\in J$ and $t \in I\setminus J$, we have $m:=(a_j-b_j)-(a_t-b_t)> 0$ and thus
		$$\Hom_{\X}(U_t(a_t\w),U_j(a_j\w))=\Hom_{\X}(U_t(b_t\w),U_j((m+b_j)\w))=0$$ by the definition of slice $V_2$.   We  number the indecomposable direct summands
		of $V_1'$ as $V_1'=U_1'\oplus\dots\oplus U_{\ell}'$ such that $\Hom(U_{i}',U_{j}')=0$ for any $i>j$. 
		 Therefore we may mutate and obtain $$\mu^+_{U_\ell'}\circ\dots\circ\mu^+_{U_1'} (V_1) =\mu^+_{V_1'}(V_1)=\bigoplus_{j\in J} V_j((a_j-1)\w) \oplus \bigoplus_{t\notin J} V_t(a_t\w).$$
		Repeating this procedure,   $V_1$ and $V_2$ are iterated  mutations of each other.
	\end{proof}
	
	When convenient, 	we denote ${\Ext}^{i}_{\coh \X}(-,-)$ by ${\Ext}^{i}(-,-)$, even more simplified by ${}^{i}(-,-)$ for $i\ge 0$, and denote ${\Hom}_{\underline{\CM}^{\L} R}(-,-)$ by  $\underline{\Hom}(-,-)$.
	
		\section{Proof of Theorem \ref{main theorem}}\label{sec:main results}
	%\section{Main results} \label{sec:main results}		
		Throughout this section, we assume $d=2$, $n=4$ and $(R,\L)$ a GL complete intersection of  type $(2,2,p,q)$ with integers $p,q \ge 2$, and $\X$ the corresponding GL projective space. Recall from \cite[Observation 3.3]{HIMO} that $R$ can be normalized as a hypersurface singularity
		$$R=k[X_1, \dots, X_4]/(X_1^{2}+X_2^{2}+X_3^{p}+X_4^{q}).$$ 
		Let $\s:=\sum_{i=1}^{4}\x_i$, $\w:=\c-\s$ and $\vdelta :=2\c+2\w=(p-2)\x_3+(q-2)\x_4$.  For each $\vell=\sum_{i=1}^{4}\ell_i\x_i$  in the internal $[\s,\s+\vdelta]$, we let
		$$U^{\vell}:=\rho(R/(X_i^{\ell_i}\mid 1\le i\le 4))\ \text{ and } \ U^{\rm CM}:=\bigoplus_{ \vell \in [\s,\s+\vdelta]}U^{\vell},$$ 
		where $\rho$ denotes the composition $\DDD^{\bo}(\mod^{\L}R)\to\DDD_{\rm sg}^{\L}(R)\xrightarrow{\sim} \underline{\CM}^{\L}R$ and the set $\mathcal{S}$  consists of elements $\x=\sum_{i=1}^{4}\lambda_i\x_i+\lambda\c$ in normal form satisfying (\ref{S=}).

	\begin{proposition} \label{pre-prosition}
		{\rm (See \cite[Theorem 4.76, Corollary 7.32]{HIMO})}  
		\begin{itemize}			
			\item[(a)]   $U^{\rm CM}$ is a $2$-tilting object in $\underline{\CM}^{\L}R$.
			\item[(b)] The object \[T:=\pi(U^{\rm CM})\oplus(\bigoplus_{\x\in \mathcal{S}}\OO(\x))\]
			 is a $2$-tilting bundle on $\X$.	
		\end{itemize}
		 	
	\end{proposition}

		The following  result presents a family of  $2$-tilting bundles  on $\X$ by the action of  iterated $2$-APR mutations to the $2$-tilting bundle $T$ in the preceding proposition.  Let $J$ be an upset restricted to $[\s,\s+\vdelta]$ and $I:=[\s,\s+\vdelta]\setminus J$.
		We put
		\[M:=(\bigoplus_{\vell\in I} U^{\vell} ) \oplus  (\bigoplus_{\z\in J} U^{\z}(-\w) ) \oplus(\bigoplus_{\x\in \mathcal{S}}R(\x)).\]
		
	\begin{theorem}\label{main theorem_1} 		
		   		  $\pi(M)$ is a $2$-tilting bundle on $\X$. 
	
	\end{theorem}	
		
		 To prove Theorem \ref{main theorem_1}, we need the following observations. By  a similar argument  of \cite[Lemma 3.5]{I3}, we have the following lemma.  Here, $\add X$ denotes the full subcategory of $\coh \X$ consisting of direct summands of finite direct sums of $X$.
		\begin{lemma}\label{applying hom} Let $X\in \coh \X$ and $$0\to X_0 \to X_1\to X_2\to X_{3} \to 0$$
			an exact sequence in $\coh \X$ with $X_i\in \add X$.
			\begin{itemize}			
				\item[(a)]   If $W\in \coh\X$ satisfies $\Ext_{}^{1}(W,X)=0$, then there exists an exact sequence
				$$0\to (W,X_{0})\to\cdots\to(W,X_{3})\to {}^2(W,X_{0})\to\cdots\to{}^2(W,X_{3})\to 0.$$
				\item[(b)] If $Y\in \coh\X$ satisfies $\Ext_{}^{1}(X,Y)=0$, then there exists an exact sequence
				$$0\to(X_3,Y)\to\cdots\to(X_0,Y) \to{}^2(X_3,Y)\to\cdots\to{}^2(X_0,Y)\to0.$$
			\end{itemize}
		\end{lemma}

		The following  proposition  was shown by a general result in \cite[Theorem 5.1]{CRW}.
		Here, we put $(p_1, p_2, p_3, p_4):=(2,2,p,q)$.
		
		\begin{proposition}\label{2-extension bundle} For each $\vell \in [\s, \s+\vdelta]$, there exists an %non-split 
			exact sequence 
			\begin{equation}\label{org 2-ext bundle exact}
				\eta: \quad	0 \to \OO \xrightarrow{} 
				\pi(U^{\vell}) 
				\xrightarrow{} 
				\bigoplus_{1\le i \le 4}\OO(2\c-\vell+\ell_{i}\x_i) \xrightarrow{\gamma} \OO(3\c-\vell) \to 0
			\end{equation}
			in $\coh \X$, where $\vell:=\sum_{i=1}^{4} \ell_i\x_i$ and $\gamma:=(X_i^{p_i-\ell_i})_{1\le i\le 4}$.
		\end{proposition}
		
	Recall from  \cite[Proposition 5.19(b)]{HIMO} that for any $\x\in\L$ and $\vell\in [\s, \s+\vdelta]$,  $$\Ext^{1}(\pi(U^{\vell}),\OO(\x))=0 \text{ and } \Ext^{1}(\OO(\x),\pi(U^{\vell}))=0.$$

	\begin{proposition} \label{lem2}
		For any $\s \le \vell, \z \le \s+\vdelta$ with $\vell \not\ge \z$, then 
		$\Hom(\pi(U^{\vell}), \pi(U^{\z}))=0.$ 		
	\end{proposition}
	\begin{proof} By Proposition \ref{2-extension bundle}, there exists an  exact sequence
		\begin{align*}
			\eta_{\vell}: & \quad	0 \to \OO \xrightarrow{} 
		\pi(U^{\vell}) 
		\xrightarrow{} 
		\bigoplus_{1\le i \le 4}\OO(2\c-\vell+\ell_{i}\x_i) \xrightarrow{\gamma_{\vell}} \OO(3\c-\vell) \to 0,	
		\end{align*}
		where $\vell:=\sum_{i=1}^{4}\ell_i\x_i$ and $\gamma_{\vell}:=(X_i^{p_i-\ell_i})_{1\le i\le 4}$. Write $\z:=\sum_{i=1}^{4}z_i\x_i$. 
	We factor the sequence  $\eta_{\z}$ into the two short exact sequences
	\begin{align*} 
		\eta_{\z,1}: \quad	0 &\to \OO \to \pi(U^{\z}) \to K \to 0,\\
		\eta_{\z,2}: \quad	0&\to K \to \bigoplus_{1\le i\le 4}\OO(2\c-\z+z_{i}\x_i) \xrightarrow{\gamma_{\z}} \OO(3\c-\z) \to 0.
	\end{align*}	The proof is divided into the following three steps.
		
			\emph{Step $1$}: We show that $\Hom(\pi(U^{\vell}),\OO)=0$. By Lemma \ref{applying hom}, applying $\Hom(-,\OO)$ to the sequence $\eta_{\vell}$ and obtain an exact sequence
		\begin{align} \label{equ 1.}
			\bigoplus_{1\le i\le 4}(\OO(2\c-\vell+\ell_{i}\x_i), \OO)
			\to (\pi(U^{\vell}),\OO) \to (\OO,\OO)\xrightarrow{\partial} {}^{2}(\OO(3\c-\vell),\OO).
		\end{align}
		Since $-2\c+\vell-\ell_{i}\x_i \le 2\c+\w$, we have that the leftmost term in (\ref{equ 1.})  is zero. Using Auslander-Reiten-Serre  duality,  we have $$\Ext^{2}(\OO(3\c-\vell),\OO)=D\Hom(\OO,\OO(3\c-\vell+\w))=k.$$ 
		Since $\Hom(\OO,\OO)=k$, the  connecting morphism $\partial$ is nonzero, which implies that $\partial$ is an isomorphism. Thus we have $\Hom(\pi(U^{\vell}),\OO)=0.$
		
			\emph{Step $2$}:	We show that $\Hom(\pi(U^{\vell}),K)=0$.		
		 By Lemma \ref{applying hom}, for $1\le i \le 4$,  we apply $\Hom(-,\OO(2\c-\z+z_{i}\x_i))$ to the sequence $\eta_{\vell}$ and obtain an exact sequence
		\begin{align*}
			0 &\to (\OO(3\c-\vell),\OO(2\c-\z+z_{i}\x_i)) \to \bigoplus_{1\le j\le 4}(\OO(2\c-\vell+\ell_{j}\x_j), \OO(2\c-\z+z_{i}\x_i))\\ 
			&\to (\pi(U^{\vell}),\OO(2\c-\z+z_{i}\x_i)) \to (\OO,\OO(2\c-\z+z_{i}\x_i)).
		\end{align*} 
		Clearly, we have $\Hom(\OO,\OO(2\c-\z+z_{i}\x_i))=0$. By our assumption $\vell \not\ge \z$, we have $\vell- \z \le 2\c+\w$ by (\ref{two possibilities of x}) and thus 
		$2\c-\z+z_{i}\x_i-3\c+\vell\le 2\c+\w$, implying $\Hom(\OO(3\c-\vell),\OO(2\c-\z+z_{i}\x_i))=0$. 		 This implies that
		\begin{align} \label{align 2}
			(\pi(U^{\vell}),\OO(2\c-\z+z_{i}\x_i)) \simeq	\bigoplus_{1\le j\le 4} (\OO(2\c-\vell+\ell_{j}\x_j), \OO(2\c-\z+z_{i}\x_i)).
		\end{align} 
		Applying $\Hom(\pi(U^{\vell}),-)$ to  $\eta_{\z,2}$, we have an exact sequence
		\begin{align*} 
			0 \to (\pi(U^{\vell}),K) \to \bigoplus_{1\le i\le 4}(\pi(U^{\vell}),\OO(2\c-\z+z_{i}\x_i)) \xrightarrow{\gamma_{\ast}} (\pi(U^{\vell}),\OO(3\c-\z)).
		\end{align*}
	    Notice that the inequality $2\c-\z+z_{i}\x_i-2\c+\vell-\ell_{j}\x_j \ge 0$, which 
		 is equivalent to  $$\sum_{t\neq j}(\ell_t-z_t)\x_t+z_i\x_i-z_j\x_j \ge 0.$$ This holds if and only if $i=j$ and $\ell_t\ge z_t$ for any $t \neq j$. %Thus the claim follows.
		Thus by (\ref{align 2}), we have
		\begin{align} \label{align 3}
			\bigoplus_{1\le i\le 4}(\pi(U^{\vell}),\OO(2\c-\z+z_{i}\x_i))\neq 0
		\end{align}
		 if and only if  there exists a subset $I$ of $H:=\{1,\dots,4\}$ with $|I|=3$ such that $\ell_i\ge z_i$ for any $i \in I$ and $\ell_j <z_j$ for $j\in H\setminus I$ since $\vell \not\ge \z$. In this case, it is isomorphic to $k$.	We may assume that (\ref{align 3}) holds, otherwise we have done.	
		Applying $\Hom(-,\OO(3\c-\z))$ to the sequence $\eta_{\vell}$, we obtain an exact sequence  
		\begin{align*}
			0 &\to (\OO(3\c-\vell),\OO(3\c-\z)) \to \bigoplus_{1\le i\le 4}(\OO(2\c-\vell+\ell_{i}\x_i), \OO(3\c-\z))\\ 
			&\to (\pi(U^{\vell}),\OO(3\c-\z)) \to (\OO,\OO(3\c-\z))=0.
		\end{align*}
		Since $\vell\not\ge\z$, we have $\Hom(\OO(3\c-\vell),\OO(3\c-\z))=0$, implying 
		\begin{align} \label{equ2.}
			\Hom(\pi(U^{\vell}),\OO(3\c-\z)) \simeq	\bigoplus_{1\le i\le 4} \Hom(\OO(2\c-\vell+\ell_{i}\x_i), \OO(3\c-\z)).
		\end{align}
		By the assumption, 	there exists a subset $I$ of $H=\{1,\dots,4\}$ with $|I|=3$ such that $\ell_i=z_i$ for any $i \in I$ and $\ell_j <z_j$ for $j\in H\setminus I$. In this case, there is exactly one $i$ such that $3\c-\z-2\c+\vell-\ell_{i}\x_i\ge 0$ and then (\ref{equ2.}) is isomorphic to $k$. Thus the
		nonzero morphism $\gamma_{\ast}$ between one dimensional spaces is an isomorphism, implying that $\Hom(\pi(U^{\vell}), K)=0$.
		
		\emph{Step $3$}: Applying $\Hom(\pi(U^{\vell}),-)$ to  $\eta_{\z,1}$, we  obtain an exact sequence 
		$$0 \to (\pi(U^{\vell}), \OO) \to (\pi(U^{\vell}), \pi(U^{\z})) \to (\pi(U^{\vell}), K)\to 0.$$
		By Step $1$ and $2$, we have $\Hom(\pi(U^{\vell}), \pi(U^{\z}))=0$ and this finishes the proof.
	\end{proof}

	Now we are ready to prove Theorem \ref{main theorem_1}.	
		
	\begin{proof}[Proof of Theorem~\rm\ref{main theorem_1}] 	
	First we claim that  $\Hom(\OO(\x),\pi(U^{\vell}))=0$ holds for any  $\x\in \mathcal{S}$ and $\vell \in [\s,\s+\vdelta]$. By Lemma \ref{applying hom},  applying $\Hom(\OO(\x),-)$ to the exact sequence (\ref{org 2-ext bundle exact}), we obtain an exact sequence
	$$  (\OO(\x),\OO) \to (\OO(\x),\pi(U^{\vell})) \to \bigoplus_{1\le i\le 4}(\OO(\x),\OO(2\c-\vell+\ell_{i}\x_i)) \xrightarrow{\gamma_{\ast}} (\OO(\x),\OO(3\c-\vell)).$$
	Write $\x=\sum_{i=1}^{n}\lambda_i\x_i+\lambda\c$ in the normal form. Note that we have $-1\le \lambda \le 1$.  Concerning the value of $\lambda$, the proof of the claim is divided into the  three cases.
	
	\emph{Case $1$}: In the case $\lambda=1$, and thus $\x=\c$. By (\ref{two possibilities of x}), it is straightforward to check that $\Hom(\OO(\c),\OO)=0$ and $\Hom(\OO(\c),\OO(2\c-\vell+\ell_{i}\x_i))=0$ for each $1\le i\le 4$, which implies that $\Hom(\OO(\x),\pi(U^{\vell}))=0$.
	
	\emph{Case $2$}: In the case $\lambda=0$, and thus $\x=\sum_{i=1}^{4}\lambda_i\x_i$. Since $\x \in \mathcal{S}$, we have 
	$0 < \# \{1\le i \le 4\mid\lambda_i\ge1\}$, yielding $\x >0$ and thus $\Hom(\OO(\x),\OO)=0$. On the other hand, since $2\c-\vell+\ell_{i}\x_i -\x \le 2\c+\w$, we have $\Hom(\OO(\x),\OO(2\c-\vell+\ell_{i}\x_i))=0$ for each $1\le i\le 4$. This implies that 
	$\Hom(\OO(\x),\pi(U^{\vell}))=0$.
	
	\emph{Case $3$}: In the case $\lambda=-1$, and thus $\x=\sum_{i=1}^{4}\lambda_i\x_i-\c$.  
	Since $\x \in \mathcal{S}$, we have $2 < \# \{1\le i \le 4\mid\lambda_i\ge1\}$, which yields that $2\c+\sum_{i=1}^{4}(\lambda_i-1)\x_i \ge 0$. Thus we have $-\x\le 2\c+\w$, implying $\Hom(\OO(\x),\OO)=0$. On the other hand, if $2\c-\vell+\ell_{i}\x_i-\x \not\ge 0$ holds for all $1\le i \le 4$, then  $\Hom(\OO(\x),\OO(2\c-\vell+\ell_{i}\x_i))=0$ for $1\le i \le 4$ and thus $\Hom(\OO(\x),\pi(U^{\vell}))=0$. 			
	Otherwise 
	$2\c-\vell+\ell_{t}\x_t-\x \ge 0$ holds for some $1\le t \le 4$, by a easy calculation, which is equivalent to 
	\begin{align*} \label{align 1}
		\sum_{j\neq t} (\ell_j+\lambda_j)\x_j+\lambda_t\x_t \le 3\c,
	\end{align*}
	yielding $\lambda_t=0$ and $\lambda_j+\ell_j\le p_j$ for any $j\neq t$. 
	 Since $2 < \# \{1\le i \le 4\mid\lambda_i\ge1\}$, this implies that  $\lambda_j \ge 1$ for any $j\neq t$. Therefore  $2\c-\vell+\ell_{i}\x_i-\x \ge 0$ is only possible for $i=t$. 
	In this case, we have
	$$\bigoplus_{1\le i\le 4}\Hom(\OO(\x),\OO(2\c-\vell+\ell_{i}\x_i))=k \ \text{ and } \  \Hom(\OO(\x), \OO(3\c-\vell))=k,$$
	which shows that the nonzero morphism $\gamma_{\ast}$ is an isomorphism. This also implies that $\Hom(\OO(\x),\pi(U^{\vell}))=0$ and thus the claim follows.
	
	Combining the preceding claim and Proposition \ref{lem2}, we have that for any $\x \in \mathcal{S}$, $\vell\in I$ and $\z \in J$, $\Hom(\pi(U^{\vell})\oplus \OO(\x), \pi(U^{\z}))=0.$ Hence the  assertion follows immediately from Lemma \ref{mutation again}.
	\end{proof} 	
	 
	 As an application, we find some  non-examples of  Problem \ref{problem}. From now on, we focus on a GL hypersurface $(R,\L)$ of  type $(2,2,2,q)$ or $(2,2,3,3)$, where $q \ge 2$.   Such a GL hypersurface is of Cohen-Macaulay finite type, as discussed in \cite{HIMO}.
	 
	 \subsection*{Case $(2,2,2,q)$} In this case, there exists a triangle equivalence
	 \[ \underline{\CM}^{\L} R  \simeq \DDD^{\bo}(\mod k \mathbb A_{q-1}).\]
	 Identify objects in $\underline{\CM}^{\L}R$ with objects in ${\CM}^{\L}R$ without nonzero projective direct summands. By adding all indecomposable projective modules $R(\x)$ for $\x \in \L$ and arrows connected to them into the Auslander-Reiten quiver  ${\mathfrak A}(\underline{\CM}^{\L} R)$, we obtain the following Auslander-Reiten quiver ${\mathfrak A}({\CM}^{\L} R)$, see \cite[Theorems 4.13, 4.46]{HIMO}. 
		\[ \resizebox{\textwidth}{!}{
			\begin{xy} 0;<16pt,0pt>:<0pt,16pt>::
			(16,8) *+{\cdot} ="42",
			(18,10) *+{\bullet} ="52",
			(16,4) *+{} ="23",
			(16.4,4.5) *+{\mdots} ="x",
			(17.2,5.4) *+{\mdots} ="y",
			(18,6) *+{} ="33",
			(20,8) *+{\cdot} ="43",
			(22,10) *+{\bullet} ="53",
			(16,0) *+{\bullet} ="04",
			(18,2) *+{\cdot} ="14",
			(20,4) *+{} ="24",
			(20.4,4.5) *+{\mdots} ="x",
			(21.2,5.4) *+{\mdots} ="y",
			(22,6) *+{} ="34",
			(24,8) *+{\cdot} ="44",
			(26,10) *+{\bullet} ="54",
			(20,0) *+{\bullet} ="05",
			(22,2) *+{\cdot} ="15",
			(24,4) *+{} ="25",
			(24.4,4.5) *+{\mdots} ="x",
			(25.2,5.4) *+{\mdots} ="y",
			(26,6) *+{} ="35",
			(28,8) *+{\cdot} ="45",
			(30,10) *+{\bullet} ="55",
			(24,0) *+{(0)} ="06",
			(26,2) *+{\langle 0,0 \rangle} ="16",
			(28,4) *+{} ="26",
			(28.4,4.5) *+{\mdots} ="x",
			(29.2,5.4) *+{\mdots} ="y",
			(30,6) *+{} ="36",
			(32,8) *+{\langle q-2, 0 \rangle} ="46",
			(34,10) *+{(\x_{3})} ="56",
			(28,0) *+{(\x_4)} ="07",
			(30,2) *+{\langle 0,1 \rangle} ="17",
			(32,4) *+{} ="27",
			(32.4,4.5) *+{\mdots} ="x",
			(33.2,5.4) *+{\mdots} ="y",
			(34,6) *+{} ="37",
			(36,8) *+{\langle q-2, 1 \rangle} ="47",
			(38,10) *+{(\x_{3}+\x_4)} ="57",
			(32,0) *+{\bullet} ="08",
			(34,2) *+{\cdot} ="18",
			(36,4) *+{} ="28",
			(36.4,4.5) *+{\mdots} ="x",
			(37.2,5.4) *+{\mdots} ="y",
			(38,6) *+{} ="38",
			(40,8) *+{\cdot} ="48",
			(42,10) *+{\bullet} ="58",
			(36,0) *+{\bullet} ="09",
			(38,2) *+{\cdot} ="19",
			(40,4) *+{} ="29",
			(40.4,4.5) *+{\mdots} ="x",
			(41.2,5.4) *+{\mdots} ="y",
			(42,6) *+{} ="39",
			(44,8) *+{\cdot} ="49",
			(46,10) *+{\bullet} ="59",
			(40,0) *+{\bullet} ="010",
			(42,2) *+{\cdot} ="110",
			(44,4) *+{} ="210",
			(44.4,4.5) *+{\mdots} ="x",
			(45.2,5.4) *+{\mdots} ="y",
			(46,6) *+{} ="310",
			(44,0) *+{\bullet} ="011",
			(46,2) *+{\cdot} ="111",
			"42", {\ar"52"},
			"42", {\ar"33"},
			"52", {\ar"43"},
			"33", {\ar"43"},
			"43", {\ar"53"},
			%"13", {\ar"04"},
			"23", {\ar"14"},
			"43", {\ar"34"},
			"53", {\ar"44"},
			"04", {\ar"14"},
			"14", {\ar"24"},
			"34", {\ar"44"},
			"44", {\ar"54"},
			"14", {\ar"05"},
			"24", {\ar"15"},
			"44", {\ar"35"},
			"54", {\ar"45"},
			"05", {\ar"15"},
			"15", {\ar"25"},
			"35", {\ar"45"},
			"45", {\ar"55"},
			"15", {\ar"06"},
			"25", {\ar"16"},
			"45", {\ar"36"},
			"55", {\ar"46"},
			"06", {\ar"16"},
			"16", {\ar"26"},
			"36", {\ar"46"},
			"46", {\ar"56"},
			"16", {\ar"07"},
			"26", {\ar"17"},
			"46", {\ar"37"},
			"56", {\ar"47"},
			"07", {\ar"17"},
			"17", {\ar"27"},
			"37", {\ar"47"},
			"47", {\ar"57"},
			"17", {\ar"08"},
			"27", {\ar"18"},
			"47", {\ar"38"},
			"57", {\ar"48"},
			"08", {\ar"18"},
			"18", {\ar"28"},
			"38", {\ar"48"},
			"48", {\ar"58"},
			"18", {\ar"09"},
			"28", {\ar"19"},
			"48", {\ar"39"},
			"58", {\ar"49"},
			"09", {\ar"19"},
			"19", {\ar"29"},
			"39", {\ar"49"},
			"49", {\ar"59"},
			"19", {\ar"010"},
			"29", {\ar"110"},
			"49", {\ar"310"},
			"010", {\ar"110"},
			"110", {\ar"210"},
			"110", {\ar"011"},
			"210", {\ar"111"},
			%"410", {\ar"311"},
			%
			"011", {\ar"111"},
			\end{xy}
		} \]
	    	Here, 
	    	a vertex labeled $\langle i,j \rangle$ denotes the position of  $U^{\s+(q-2-i)\x_4}(j\x_4)$ for $0\le i \le q-2$ and $j\in\Z$. The 
	    	position of $R(\x)$ is denoted $(\x)$ or $\bullet$ in case $\x$ is not specified. The  
	    	positions of $R(\x)$ and  $R(\x+\vec{t}_{ab})$ are overlapping  for integers $1\le a < b \le 3$, where $\vec{t}_{ab}:=\x_a-\x_b$. The interpretation is that every arrow starting or ending at vertex $(\x)$ or $\bullet$ represents four arrows in the actual Auslander-Reiten quiver (each one connected to one of the $R(\x)$ and $R(\x+\vec{t}_{ab})$).

	 \begin{proposition} \label{pro1} Assume that $(R,\L)$ is given by weight type $(2,2,2,q)$ with $q\ge 2$.  Then  $M$  is not a $2$-tilting object  in $\underline{\CM}^{\L} R$ if and only if $J=[\s+m\x_4,\s+(q-2)\x_4]$ for some $1\le m \le q-2$.	 	
	 \end{proposition}	
	 \begin{proof}
	 First we assume for contradiction that $J=\emptyset$ or $[\s,\s+(q-2)\x_4]$.  
	By Proposition \ref{pre-prosition}, $U^{\rm CM}:= \bigoplus_{\s \le \vell \le \s+(q-2)\x_4} U^{\vell}$ is a $2$-tilting object in $\underline{\CM}^{\L} R$ and also $U^{\rm CM}(-\w)$ is a $2$-tilting object in $\underline{\CM}^{\L} R$.  
		
	On the other hand,  assume 
	$J=[\s+m\x_4,\s+(q-2)\x_4]$ for some  $1\le m \le q-2$ and thus $I=[\s,\s+(m-1)\x_4]$. 	
	Notice that the irreducible morphism $U^{\s+m\x_4} \to U^{\s+(m-1)\x_4}$ induces the irreducible morphism $U^{\s+(m-1)\x_4} \to \tau^{-1} U^{\s+m\x_4}$, where $\tau$ denotes the Auslander–Reiten translation. By \cite[Proposition 4.28]{HIMO}, 
	we have
	\begin{align*}
		\underline{\Hom}(U^{\s+(m-1)\x_4},U^{\s+m\x_4}(-\w)[-1])
		=\underline{\Hom}(U^{\s+(m-1)\x_4},\tau^{-1}U^{\s+m\x_4})=k,
	\end{align*}
	 and thus $M$ is not a tilting object in $\underline{\CM}^{\L} R$. Hence the assertion follows.
	 \end{proof}
		
	\subsection*{Case $(2,2,3,3)$} In this case, there exists a triangle equivalence
	\[ \underline{\CM}^{\L} R  \simeq \DDD^{\bo}(\mod k \mathbb D_{4}).\]	
	 Adding all indecomposable projective modules $R(\x)$ for $\x \in \L$ and arrows connected to them into the  Auslander-Reiten quiver  ${\mathfrak A}(\underline{\CM}^{\L} R)$, we obtain the Auslander-Reiten quiver ${\mathfrak A}({\CM}^{\L} R)$, which is depicted below, see \cite[Theorems 4.13, 4.46]{HIMO}.  
	 	\[
	\resizebox{\textwidth}{!}{
		\begin{xy} 0;<16pt,0pt>:<0pt,16pt>::
			(16,8) *+{\circ} ="42",
			(16,6) *+{\circ} ="62",
			(16,4) *+{\cdot} ="23",
			(18,6) *+{\cdot} ="33",
			(20,8) *+{\bullet} ="43",
			(18,5) *+{\cdot} ="53",
			(20,6) *+{\bullet} ="63",
			(16,0) *+{\circ} ="04",
			(18,2) *+{\cdot} ="14",
			(20,4) *+{\cdot} ="24",
			(22,6) *+{\cdot} ="34",
			(24,8) *+{\circ} ="44",
			(22,5) *+{\cdot} ="54",
			(24,6) *+{\circ} ="64",
			(20,0) *+{(0)} ="05",
			(22,2) *+{U^{\s+\vdelta}} ="15",
			(24,4) *+{\cdot} ="25",
			(26,6) *+{U^{\s+\x_3}} ="35",
			(28,8) *+{(\x_{3})} ="45",
			(26,4.9) *+{U^{\s+\x_4}} ="55",
			(28,6) *+{(\x_4)} ="65",
			(24,0) *+{\circ} ="06",
			(26,2) *+{\cdot} ="16",
			(28,4) *+{G} ="26",
			(30,6) *+{\cdot} ="36",
			(32,8) *+{\circ} ="46",
			(30,5) *+{\cdot} ="56",
			(32,6) *+{\circ} ="66",
			(28,0) *+{\bullet} ="07",
			(30,2) *+{U^{\s}} ="17",
			(32,4) *+{\cdot} ="27",
			(34,6) *+{\cdot} ="37",
			(36,8) *+{\bullet} ="47",
			(34,5) *+{\cdot} ="57",
			(36,6) *+{\bullet} ="67",
			(32,0) *+{(\x_1)} ="08",
			(34,2) *+{\cdot} ="18",
			(36,4) *+{\cdot} ="28",
			(38,6) *+{\cdot} ="38",
			(40,8) *+{\circ} ="48",
			(38,5) *+{\cdot} ="58",
			(40,6) *+{\circ} ="68",
			(36,0) *+{\bullet} ="09",
			(38,2) *+{\cdot} ="19",
			(40,4) *+{\cdot} ="29",
			(42,6) *+{\cdot} ="39",
			(44,8) *+{\bullet} ="49",
			(42,5) *+{\cdot} ="59",
			(44,6) *+{\bullet} ="69",
			(40,0) *+{\circ} ="010",
			(42,2) *+{\cdot} ="110",
			(44,4) *+{\cdot} ="210",
			(44,0) *+{\bullet} ="011",
			"42", {\ar"33"},
			"62", {\ar"53"},
			"23", {\ar"33"},
			"33", {\ar"43"},
			"23", {\ar"53"},
			"53", {\ar"63"},
			"23", {\ar"14"},
			"33", {\ar"24"},
			"43", {\ar"34"},
			"53", {\ar"24"},
			"63", {\ar"54"},
			"04", {\ar"14"},
			"14", {\ar"24"},
			"24", {\ar"34"},
			"34", {\ar"44"},
			"24", {\ar"54"},
			"54", {\ar"64"},
			"14", {\ar"05"},
			"24", {\ar"15"},
			"34", {\ar"25"},
			"44", {\ar"35"},
			"54", {\ar"25"},
			"64", {\ar"55"},
			"05", {\ar"15"},
			"15", {\ar"25"},
			"25", {\ar"35"},
			"35", {\ar"45"},
			"25", {\ar"55"},
			"55", {\ar"65"},
			"15", {\ar"06"},
			"25", {\ar"16"},
			"35", {\ar"26"},
			"45", {\ar"36"},
			"55", {\ar"26"},
			"65", {\ar"56"},
			"06", {\ar"16"},
			"16", {\ar"26"},
			"26", {\ar"36"},
			"36", {\ar"46"},
			"26", {\ar"56"},
			"56", {\ar"66"},
			"16", {\ar"07"},
			"26", {\ar"17"},
			"36", {\ar"27"},
			"46", {\ar"37"},
			"56", {\ar"27"},
			"66", {\ar"57"},
			"07", {\ar"17"},
			"17", {\ar"27"},
			"27", {\ar"37"},
			"37", {\ar"47"},
			"27", {\ar"57"},
			"57", {\ar"67"},
			"17", {\ar"08"},
			"27", {\ar"18"},
			"37", {\ar"28"},
			"47", {\ar"38"},
			"57", {\ar"28"},
			"67", {\ar"58"},
			"08", {\ar"18"},
			"18", {\ar"28"},
			"28", {\ar"38"},
			"38", {\ar"48"},
			"28", {\ar"58"},
			"58", {\ar"68"},
			"18", {\ar"09"},
			"28", {\ar"19"},
			"38", {\ar"29"},
			"48", {\ar"39"},
			"58", {\ar"29"},
			"68", {\ar"59"},
			"09", {\ar"19"},
			"19", {\ar"29"},
			"29", {\ar"39"},
			"39", {\ar"49"},
			"29", {\ar"59"},
			"59", {\ar"69"},
			"19", {\ar"010"},
			"29", {\ar"110"},
			"39", {\ar"210"},
			"59", {\ar"210"},
			"010", {\ar"110"},
			"110", {\ar"210"},
			"110", {\ar"011"},
		\end{xy}
	}
	\]
	Here, the position of $R(\x)$ is denoted $(\x)$, $\bullet$ or $\circ$ in case $\x$ is not specified. The  
	positions of $R(\x)$ and  $R(\x+\x_1-\x_2)$ are overlapping. Every arrow starting or ending at one of vertices $(\x)$, $\bullet$ and $\circ$ represents two arrows in the actual Auslander-Reiten quiver (each one connected to  $R(\x)$ or $R(\x+\x_1-\x_2)$).
	 \begin{proposition} \label{pro2} Assume that $(R,\L)$ is given by weight type $(2,2,3,3)$.  Then  $M$  is not  a $2$-tilting object in $\underline{\CM}^{\L} R$ if and only if $J=[\s+\x_i,\s+\x_3+\x_4]$ for $i=3,4$. 
	\end{proposition}
	\begin{proof} We assume for contradiction that $J$ is one of the four cases
		$$ \emptyset, \  \{\s+\vdelta\},  \ (\s,\s+\vdelta] \text{ and }   [\s,\s+\vdelta], $$ where $\vdelta=\x_3+\x_4$. 
		 For the case $J=\emptyset$ or $[\s,\s+\vdelta]$, by Proposition \ref{pre-prosition}, $M$ is a $2$-tilting object in $\underline{\CM}^{\L} R$. Next we only show the case for $J=\{\s+\vdelta\}$, as one can show the case for $J=(\s,\s+\vdelta]$ by a similar argument. In this case, $$M=(\bigoplus_{\vell\in [\s,\s+\vdelta)} U^{\vell} ) \oplus  (U^{\s+\vdelta}(-\w) ) \oplus(\bigoplus_{\x\in \mathcal{S}}R(\x)).$$
		 
		 \emph{Step $1$}: We show that $M$ is rigid, that is,	$\underline{\Hom}(M,M[n])=0$ for any $n\neq 0$.			By Proposition \ref{lem2} and Proposition \ref{relationship}, we have 
		 $$\Hom^{\L}_{R}(U^{\vell},U^{\s+\vdelta})=\Hom_{\X}(\pi(U^{\vell}),\pi(U^{\s+\vdelta}))=0,$$ which yields that $\underline{\Hom}(U^{\vell},U^{\s+\vdelta})=0$. Invoking Theorem \ref{pre-prosition}, we further have
		  $\underline{\Hom}(U^{\vell},U^{\s+\vdelta}[i])=0$  for any $\vell\in [\s,\s+\vdelta)$ and $i \in \Z$. By  Auslander-Reiten-Serre duality, we have $\underline{\Hom}(U^{\s+\vdelta}(-\w),U^{\vell}[i])=0$. 		  		
		On the other hand, keeping the above notations, the Auslander-Reiten quiver of $\underline{\CM}^{\L} R$ has the  form
		\[
		\resizebox{\textwidth}{!}
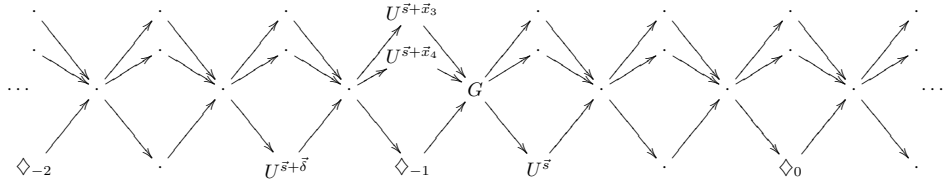
{
			\begin{xy} 0;<16pt,0pt>:<0pt,20pt>::
				(13.5,4) *+{\cdots} ="",
				(14,6) *+{\cdot} ="12",
				(14,5) *+{\cdot} ="22",				
				(14,2) *+{\diamondsuit_{-2}} ="32",
				(16,4) *+{\cdot} ="23",
				(18,6) *+{\cdot} ="33",
				(18,5) *+{\cdot} ="53",
				(18,2) *+{\cdot} ="14",
				(20,4) *+{\cdot} ="24",
				(22,6) *+{\cdot} ="34",
				(22,5) *+{\cdot} ="54",
				(22,2) *+{U^{\s+\vdelta}} ="15",
				(24,4) *+{\cdot} ="25",
				(26,6) *+{U^{\s+\x_3}} ="35",
				(26,4.9) *+{U^{\s+\x_4}} ="55",
				(26,2) *+{\diamondsuit_{-1}} ="16",
				(28,4) *+{G} ="26",
				(30,6) *+{\cdot} ="36",
				(30,5) *+{\cdot} ="56",
				(30,2) *+{U^{\s}} ="17",
				(32,4) *+{\cdot} ="27",
				(34,6) *+{\cdot} ="37",
				(34,5) *+{\cdot} ="57",
				(34,2) *+{\cdot} ="18",
				(36,4) *+{\cdot} ="28",
				(38,6) *+{\cdot} ="38",
				(38,5) *+{\cdot} ="58",
				(38,2) *+{\diamondsuit_{0}} ="19",
				(40,4) *+{\cdot} ="29",
				(42,6) *+{\cdot} ="39",
				(42,5) *+{\cdot} ="59",
				(42,2) *+{\cdot} ="110",
				(42.5,4) *+{\cdots} ="",
				"12", {\ar"23"},
				"22", {\ar"23"},
				"32", {\ar"23"},
				"23", {\ar"33"},
				"23", {\ar"53"},
				"23", {\ar"14"},
				"33", {\ar"24"},
				"53", {\ar"24"},
				"14", {\ar"24"},
				"24", {\ar"34"},
				"24", {\ar"54"},
				"24", {\ar"15"},
				"34", {\ar"25"},
				"54", {\ar"25"},
				"15", {\ar"25"},
				"25", {\ar"35"},
				"25", {\ar"55"},
				"25", {\ar"16"},
				"35", {\ar"26"},
				"55", {\ar"26"},
				"16", {\ar"26"},
				"26", {\ar"36"},
				"26", {\ar"56"},
				"26", {\ar"17"},
				"36", {\ar"27"},
				"56", {\ar"27"},
				"17", {\ar"27"},
				"27", {\ar"37"},
				"27", {\ar"57"},
				"27", {\ar"18"},
				"37", {\ar"28"},
				"57", {\ar"28"},
				"18", {\ar"28"},
				"28", {\ar"38"},
				"28", {\ar"58"},
				"28", {\ar"19"},
				"38", {\ar"29"},
				"58", {\ar"29"},
				"19", {\ar"29"},
				"29", {\ar"39"},
				"29", {\ar"59"},
				"29", {\ar"110"},
			\end{xy}
		}
		\]
		where the symbol $\diamondsuit_i$ marks the object $U^{\s+\vdelta}(-\w)[i]$ for $i\in \Z$. Next we compute the dimensions of $\underline{\Hom}(U^{\vell},-)$ for each $\vell\in [\s,\s+\vdelta)$ as follows. 						
		\[
		\resizebox{\textwidth}{!}{
			\begin{xy} 0;<16pt,0pt>:<0pt,20pt>::
				(13.5,4) *+{\cdots} ="",
				(14,6) *+{0} ="12",
				(14,5) *+{0} ="22",				
				(14,2) *+{0} ="32",
				(16,4) *+{0} ="23",
				(18,6) *+{0} ="33",
				(18,5) *+{0} ="53",
				(18,2) *+{0} ="14",
				(20,4) *+{0} ="24",
				(22,6) *+{0} ="34",
				(22,5) *+{0} ="54",
				(22,2) *+{0} ="15",
				(24,4) *+{0} ="25",
				(26,6) *+{\mathbf{1}} ="35",
				(26,5) *+{0} ="55",
				(26,2) *+{0} ="16",
				(28,4) *+{1} ="26",
				(30,6) *+{0} ="36",
				(30,5) *+{1} ="56",
				(30,2) *+{1} ="17",
				(32,4) *+{1} ="27",
				(34,6) *+{1} ="37",
				(34,5) *+{0} ="57",
				(34,2) *+{0} ="18",
				(36,4) *+{0} ="28",
				(38,6) *+{0} ="38",
				(38,5) *+{0} ="58",
				(38,2) *+{0} ="19",
				(40,4) *+{0} ="29",
				(42,6) *+{0} ="39",
				(42,5) *+{0} ="59",
				(42,2) *+{0} ="110",
				(42.5,4) *+{\cdots} ="",
				"12", {\ar"23"},
				"22", {\ar"23"},
				"32", {\ar"23"},
				"23", {\ar"33"},
				"23", {\ar"53"},
				"23", {\ar"14"},
				"33", {\ar"24"},
				"53", {\ar"24"},
				"14", {\ar"24"},
				"24", {\ar"34"},
				"24", {\ar"54"},
				"24", {\ar"15"},
				"34", {\ar"25"},
				"54", {\ar"25"},
				"15", {\ar"25"},
				"25", {\ar"35"},
				"25", {\ar"55"},
				"25", {\ar"16"},
				"35", {\ar"26"},
				"55", {\ar"26"},
				"16", {\ar"26"},
				"26", {\ar"36"},
				"26", {\ar"56"},
				"26", {\ar"17"},
				"36", {\ar"27"},
				"56", {\ar"27"},
				"17", {\ar"27"},
				"27", {\ar"37"},
				"27", {\ar"57"},
				"27", {\ar"18"},
				"37", {\ar"28"},
				"57", {\ar"28"},
				"18", {\ar"28"},
				"28", {\ar"38"},
				"28", {\ar"58"},
				"28", {\ar"19"},
				"38", {\ar"29"},
				"58", {\ar"29"},
				"19", {\ar"29"},
				"29", {\ar"39"},
				"29", {\ar"59"},
				"29", {\ar"110"},
			\end{xy}
		}
		\]	{\centering (a) The dimensions of $\underline{\Hom}(U^{\s+\x_3},-)$.
			
		}				
		\[
		\resizebox{\textwidth}{!}{
			\begin{xy} 0;<16pt,0pt>:<0pt,20pt>::
				(13.5,4) *+{\cdots} ="",
				(14,6) *+{0} ="12",
				(14,5) *+{0} ="22",				
				(14,2) *+{0} ="32",
				(16,4) *+{0} ="23",
				(18,6) *+{0} ="33",
				(18,5) *+{0} ="53",
				(18,2) *+{0} ="14",
				(20,4) *+{0} ="24",
				(22,6) *+{0} ="34",
				(22,5) *+{0} ="54",
				(22,2) *+{0} ="15",
				(24,4) *+{0} ="25",
				(26,6) *+{0} ="35",
				(26,5) *+{\mathbf{1}} ="55",
				(26,2) *+{0} ="16",
				(28,4) *+{1} ="26",
				(30,6) *+{1} ="36",
				(30,5) *+{0} ="56",
				(30,2) *+{1} ="17",
				(32,4) *+{1} ="27",
				(34,6) *+{0} ="37",
				(34,5) *+{1} ="57",
				(34,2) *+{0} ="18",
				(36,4) *+{0} ="28",
				(38,6) *+{0} ="38",
				(38,5) *+{0} ="58",
				(38,2) *+{0} ="19",
				(40,4) *+{0} ="29",
				(42,6) *+{0} ="39",
				(42,5) *+{0} ="59",
				(42,2) *+{0} ="110",
				(42.5,4) *+{\cdots} ="",
				"12", {\ar"23"},
				"22", {\ar"23"},
				"32", {\ar"23"},
				"23", {\ar"33"},
				"23", {\ar"53"},
				"23", {\ar"14"},
				"33", {\ar"24"},
				"53", {\ar"24"},
				"14", {\ar"24"},
				"24", {\ar"34"},
				"24", {\ar"54"},
				"24", {\ar"15"},
				"34", {\ar"25"},
				"54", {\ar"25"},
				"15", {\ar"25"},
				"25", {\ar"35"},
				"25", {\ar"55"},
				"25", {\ar"16"},
				"35", {\ar"26"},
				"55", {\ar"26"},
				"16", {\ar"26"},
				"26", {\ar"36"},
				"26", {\ar"56"},
				"26", {\ar"17"},
				"36", {\ar"27"},
				"56", {\ar"27"},
				"17", {\ar"27"},
				"27", {\ar"37"},
				"27", {\ar"57"},
				"27", {\ar"18"},
				"37", {\ar"28"},
				"57", {\ar"28"},
				"18", {\ar"28"},
				"28", {\ar"38"},
				"28", {\ar"58"},
				"28", {\ar"19"},
				"38", {\ar"29"},
				"58", {\ar"29"},
				"19", {\ar"29"},
				"29", {\ar"39"},
				"29", {\ar"59"},
				"29", {\ar"110"},
			\end{xy}
		}
		\]{\centering (b) The dimensions of $\underline{\Hom}(U^{\s+\x_4},-)$.
			
		}			
		\[
		\resizebox{\textwidth}{!}{
			\begin{xy} 0;<16pt,0pt>:<0pt,20pt>::
				(13.5,4) *+{\cdots} ="",
				(14,6) *+{0} ="12",
				(14,5) *+{0} ="22",				
				(14,2) *+{0} ="32",
				(16,4) *+{0} ="23",
				(18,6) *+{0} ="33",
				(18,5) *+{0} ="53",
				(18,2) *+{0} ="14",
				(20,4) *+{0} ="24",
				(22,6) *+{0} ="34",
				(22,5) *+{0} ="54",
				(22,2) *+{0} ="15",
				(24,4) *+{0} ="25",
				(26,6) *+{0} ="35",
				(26,5) *+{0} ="55",
				(26,2) *+{0} ="16",
				(28,4) *+{0} ="26",
				(30,6) *+{0} ="36",
				(30,5) *+{0} ="56",
				(30,2) *+{\mathbf{1}} ="17",
				(32,4) *+{1} ="27",
				(34,6) *+{1} ="37",
				(34,5) *+{1} ="57",
				(34,2) *+{0} ="18",
				(36,4) *+{1} ="28",
				(38,6) *+{0} ="38",
				(38,5) *+{0} ="58",
				(38,2) *+{1} ="19",
				(40,4) *+{0} ="29",
				(42,6) *+{0} ="39",
				(42,5) *+{0} ="59",
				(42,2) *+{0} ="110",
				(42.5,4) *+{\cdots} ="",
				"12", {\ar"23"},
				"22", {\ar"23"},
				"32", {\ar"23"},
				"23", {\ar"33"},
				"23", {\ar"53"},
				"23", {\ar"14"},
				"33", {\ar"24"},
				"53", {\ar"24"},
				"14", {\ar"24"},
				"24", {\ar"34"},
				"24", {\ar"54"},
				"24", {\ar"15"},
				"34", {\ar"25"},
				"54", {\ar"25"},
				"15", {\ar"25"},
				"25", {\ar"35"},
				"25", {\ar"55"},
				"25", {\ar"16"},
				"35", {\ar"26"},
				"55", {\ar"26"},
				"16", {\ar"26"},
				"26", {\ar"36"},
				"26", {\ar"56"},
				"26", {\ar"17"},
				"36", {\ar"27"},
				"56", {\ar"27"},
				"17", {\ar"27"},
				"27", {\ar"37"},
				"27", {\ar"57"},
				"27", {\ar"18"},
				"37", {\ar"28"},
				"57", {\ar"28"},
				"18", {\ar"28"},
				"28", {\ar"38"},
				"28", {\ar"58"},
				"28", {\ar"19"},
				"38", {\ar"29"},
				"58", {\ar"29"},
				"19", {\ar"29"},
				"29", {\ar"39"},
				"29", {\ar"59"},
				"29", {\ar"110"},
			\end{xy}
		}
		\]{\centering (c) The dimensions of $\underline{\Hom}(U^{\s},-).$
			
		}

		Putting things together, we have shown that $\underline{\Hom}(U^{\vell},U^{\s+\vdelta}(-\w)[i])=0$ for any $\vell\in [\s,\s+\vdelta)$ and $i\neq 0$. Hence $M$ is rigid.
		
		 \emph{Step $2$}: We show that $\thick M=\underline{\CM}^{\L} R$. Notice that there are two triangles
		 \begin{eqnarray}
		 	U^{\s+\vdelta}(-\w)[-1] \to	G \to U^{\s} \to U^{\s+\vdelta}(-\w), \label{tri 1}\\
		U^{\s+\vdelta} \to U^{\s+\x_3}\oplus U^{\s+\x_4} 	 \to G \to U^{\s+\vdelta}[1] \label{tri 2}.
		\end{eqnarray}
		Thus we have $G \in \thick M$ by (\ref{tri 1}), and further  $U^{\s+\vdelta}\in \thick M$ by (\ref{tri 2}). By Proposition \ref{pre-prosition}, $U^{\rm CM}$ is a tilting object in $\underline{\CM}^{\L} R$. Since $U^{\rm CM} \in \thick M$,  we have $\thick M= \thick U^{\rm CM}=\underline{\CM}^{\L} R$.

		\emph{Step $3$}: The endomorphism algebra of $M$ is the quiver
	$$\xymatrixrowsep{0.2in}
	\xymatrixcolsep{0.2in}
	
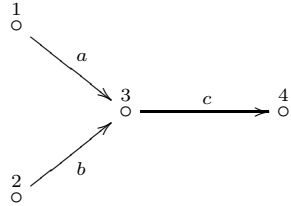
\xymatrix@C=12pt{\stackrel{1}\circ\ar[rrd]^{a}&&\\
		&&\stackrel{3}\circ\ar[rrr]^{c}&&&
		\stackrel{4}\circ~~~\\
		\stackrel{2}\circ\ar[rru]_{b}
	}$$	
	with relations $ab=0$ and $bc=0$. Thus we have gl.dim $\End (M)=2$.
	
	Therefore $M$ is a $2$-tilting object in $\underline{\CM}^{\L} R$, a contradiction to our assumption.

	On the other hand,	we only show the case for $J=[\s+\x_3,\s+\x_3+\x_4]$, as one can show the other case  similarly. By the above diagram (a), we have
	$$\Hom(U^{\s+\x_4},U^{\s+\x_3}(-\w)[-1])=\Hom(U^{\s+\x_4},\tau^{-1}U^{\s+\x_3})=k,$$
	where $\tau$ is the Auslander–Reiten translation. Thus $M$ is not a tilting object.
	\end{proof}

 	\section{examples} \label{sec:examples}
 	In this section,  we will give two  non-examples: one for the Problem \ref{problem}, and the other for the converse of  Theorem \ref{partical ans}.  
 	
 	First we present a non-example for the Problem \ref{problem}.
 	
 	\begin{example} Let $(R,\L)$ be a GL hypersurface of weight type $(2,2,2,4)$, and $\X$ be the corresponding GL projective space. Keeping the above notations, the Auslander-Reiten quiver of ${\mathfrak A}({\CM^{\L} R})$ is depicted below.	 	
 	 	\[ \resizebox{\textwidth}{!}
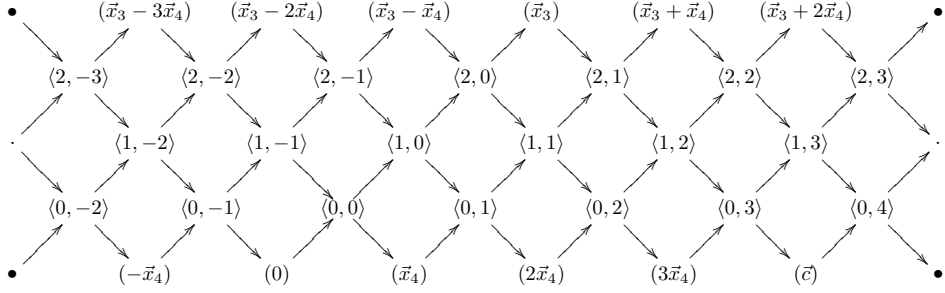
{
 		\begin{xy} 0;<16pt,0pt>:<0pt,16pt>::
 			(16,8) *+{ {\bullet}} ="42",
 			(16,4) *+{\cdot} ="23",
 			(18,6) *+{ \scalebox{1}{$\langle 2,-3 \rangle$}} ="33",
 			(20,8) *+{ \scalebox{1}{$(\x_3-3\x_4)$}} ="43",
 			(16,0) *+{ {\bullet}} ="04",
 			(18,2) *+{  \scalebox{1}{$\langle 0,-2 \rangle$}} ="14",
 			(20,4) *+{\langle 1,-2 \rangle} ="24",
 			(22,6) *+{\langle 2,-2 \rangle} ="34",
 			(24,8) *+{ \scalebox{1}{$(\x_3-2\x_4)$}} ="44",
 			(20,0) *+{ \scalebox{1}{$(-\x_4)$}} ="05",
 			(22,2) *+{\langle 0,-1 \rangle} ="15",
 			(24,4) *+{\langle 1,-1 \rangle} ="25",
 			(26,6) *+{  \scalebox{1}{$\langle 2,-1 \rangle$}} ="35",
 			(28,8) *+{ \scalebox{1}{$(\x_3-\x_4)$}} ="45",
 			(24,0) *+{\scalebox{1}{$(0)$}} ="06",
 			(26,2) *{\scalebox{1}{$\langle 0,0 \rangle$}} ="16",
 			(28,4) *+{ \scalebox{1}{$\langle 1,0 \rangle$}} ="26",
 			(30,6) *+{ \scalebox{1}{$\langle 2,0 \rangle$}} ="36",
 			(32,8) *+{ \scalebox{1}{$(\x_3)$}} ="46",
 			(28,0) *+{ \scalebox{1}{$(\x_4)$}} ="07",
 			(30,2) *+{ \scalebox{1}{$\langle 0,1 \rangle$}} ="17",
 			(32,4) *+{\langle 1,1 \rangle} ="27",
 			(34,6) *+{\langle 2,1 \rangle} ="37",
 			(36,8) *+{ \scalebox{1}{$(\x_3+\x_4)$}} ="47",
 			(32,0) *+{ (2\x_4)} ="08",
 			(34,2) *+{\langle 0,2 \rangle} ="18",
 			(36,4) *+{\langle 1,2 \rangle} ="28",
 			(38,6) *+{\scalebox{1}{$\langle 2,2 \rangle$}} ="38",
 			(40,8) *+{ \scalebox{1}{$(\x_3+2\x_4)$}} ="48",
 			(36,0) *+{ \scalebox{1}{$(3\x_4)$}} ="09",
 			(38,2) *+{ \scalebox{1}{$\langle 0,3 \rangle$}} ="19",
 			(40,4) *+{\langle 1,3 \rangle} ="29",
 			(42,6) *+{\langle 2,3 \rangle} ="39",
 			(44,8) *+{\bullet} ="49",
 			(40,0) *+{ \scalebox{1}{$(\c)$}} ="010",
 			(42,2) *+{\langle 0,4 \rangle} ="110",
 			(44,4) *+{\cdot} ="210",
 			(44,0) *+{\bullet} ="011",
 			"42", {\ar"33"},
 			"33", {\ar"43"},
 			"23", {\ar"14"},
 			"23", {\ar"33"},
 			"33", {\ar"24"},
 			"43", {\ar"34"},
 			"04", {\ar"14"},
 			"14", {\ar"24"},
 			"34", {\ar"44"},
 			"14", {\ar"05"},
 			"24", {\ar"15"},
 			"24", {\ar"34"},
 			"34", {\ar"25"},
 			"44", {\ar"35"},
 			"05", {\ar"15"},
 			"15", {\ar"25"},
 			"35", {\ar"45"},
 			"35", {\ar"26"},
 			"15", {\ar"06"},
 			"25", {\ar"16"},
 			"25", {\ar"35"},
 			"45", {\ar"36"},
 			"06", {\ar"16"},
 			"16", {\ar"26"},
 			"36", {\ar"46"},
 			"36", {\ar"27"},
 			"16", {\ar"07"},
 			"26", {\ar"17"},
 			"26", {\ar"36"},
 			"46", {\ar"37"},
 			"07", {\ar"17"},
 			"17", {\ar"27"},
 			"37", {\ar"47"},
 			"37", {\ar"28"},
 			"27", {\ar"37"},
 			"17", {\ar"08"},
 			"27", {\ar"18"},
 			"47", {\ar"38"},
 			"08", {\ar"18"},
 			"18", {\ar"28"},
 			"28", {\ar"38"},
 			"38", {\ar"48"},
 			"38", {\ar"29"},
 			"18", {\ar"09"},
 			"28", {\ar"19"},
 			"48", {\ar"39"},
 			"09", {\ar"19"},
 			"19", {\ar"29"},
 			"39", {\ar"49"},
 			"39", {\ar"210"},
 			"19", {\ar"010"},
 			"29", {\ar"110"},
 			"29", {\ar"39"},
 			"010", {\ar"110"},
 			"110", {\ar"210"},
 			"110", {\ar"011"},
 		\end{xy}
 	} \]

 	 The object $U^{\rm CM}=\bigoplus_{i=0}^{2} \langle i,0 \rangle$ is a $2$-tilting object  in $\underline{\CM}^{\L} R$, and thus $\UU:=\add \{ U^{\rm CM}(\ell\w),~R(\x) \mid \ell \in\Z,~\x\in \L \}$ is a $2$-cluster tilting subcategory of ${\CM}^{\L} R$. Note that $\langle i,j \rangle(-\w)=\langle i,j+1 \rangle[1]$, where $[1]$ denotes the suspension 
 	  in $\underline{\CM}^{\L} R$. One can compute the following Auslander-Reiten quiver of $\UU$ from ${\mathfrak A}({\CM^{\L} R})$.   	 
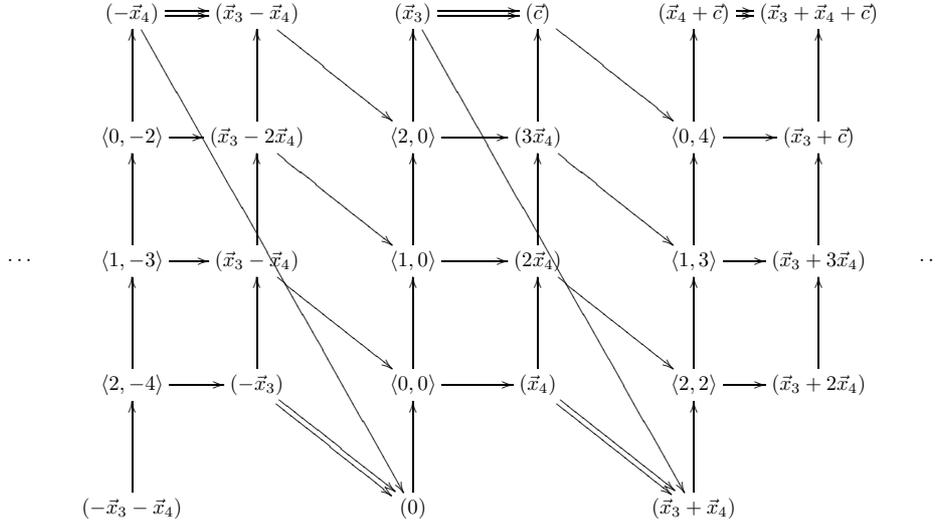
 	 \[
 	 \resizebox{\textwidth}{!}
 	 {
 	 	\begin{xy} 0;<0pt,30pt>:<30pt,0pt>:: 
 	 		(4,-1.8) *+{\cdots} ="x",
 	 		(4,12.8) *+{\cdots} ="x",
 	 		(0,0) *+{(-\x_3-\x_4)} ="00",
 	 		(2,0) *+{\langle 2,-4\rangle} ="10",
 	 		(4,0) *+{\langle 1,-3\rangle} ="20",
 	 		(6,0) *+{\langle 0,-2\rangle} ="30",
 	 		(8,0) *+{(-\x_4)} ="40",
 	 		(2,2) *+{(-\x_3)} ="11",
 	 		(4,2) *+{(\x_3-\x_4)} ="21",
 	 		(6,2) *+{(\x_3-2\x_4)} ="31",
 	 		(8,2) *+{(\x_3-\x_4)} ="41",
 	 		(0,4.5) *+{(0)} ="02",
 	 		(2,4.5) *+{\langle 0,0\rangle} ="12",
 	 		(4,4.5) *+{\langle 1,0\rangle} ="22",
 	 		(6,4.5) *+{\langle 2,0\rangle} ="32",
 	 		(8,4.5) *+{(\x_3)} ="42",
 	 		(2,6.5) *+{(\x_4)} ="13",
 	 		(4,6.5) *+{(2\x_4)} ="23",
 	 		(6,6.5) *+{(3\x_4)} ="33",
 	 		(8,6.5) *+{(\c)} ="43",
 	 		(0,9) *+{(\x_3+\x_4)} ="04",
 	 		(2,9) *+{\langle 2,2\rangle} ="14",
 	 		(4,9) *+{\langle 1,3\rangle} ="24",
 	 		(6,9) *+{\langle 0,4\rangle} ="34",
 	 		(8,9) *+{(\x_4+\c)} ="44",
 	 		(2,11) *+{(\x_3+2\x_4)} ="15",
 	 		(4,11) *+{(\x_3+3\x_4)} ="25",
 	 		(6,11) *+{(\x_3+\c)} ="35",
 	 		(8,11) *+{(\x_3+\x_4+\c)} ="45",
 	 		"00", {\ar"10"},
 	 		"10", {\ar"20"},
 	 		"20", {\ar"30"},
 	 		"30", {\ar"40"},
 	 		"11", {\ar"21"},
 	 		"21", {\ar"31"},
 	 		"31", {\ar"41"},
 	 		"02", {\ar"12"},
 	 		"12", {\ar"22"},
 	 		"22", {\ar"32"},
 	 		"32", {\ar"42"},
 	 		"13", {\ar"23"},
 	 		"23", {\ar"33"},
 	 		"33", {\ar"43"},
 	 		"04", {\ar"14"},
 	 		"14", {\ar"24"},
 	 		"24", {\ar"34"},
 	 		"34", {\ar"44"},
 	 		"15", {\ar"25"},
 	 		"25", {\ar"35"},
 	 		"35", {\ar"45"},
 	 		"10", {\ar"11"},
 	 		"20", {\ar"21"},
 	 		"30", {\ar"31"},
 	 		"40", {\ar@<-1.2pt>"41"},
 	 		"40", {\ar@<1.2pt>"41"},
 	 		"12", {\ar"13"},
 	 		"22", {\ar"23"},
 	 		"32", {\ar"33"},
 	 		"42", {\ar@<-1.2pt>"43"},
 	 		"42", {\ar@<1.2pt>"43"},
 	 		"14", {\ar"15"},
 	 		"24", {\ar"25"},
 	 		"34", {\ar"35"},
 	 		"44", {\ar@<-1.2pt>"45"},
 	 		"44", {\ar@<1.2pt>"45"},
 	 		"11", {\ar@<-.3ex>"02"},
 	 		"11", {\ar@<.3ex>"02"},
 	 		"21", {\ar"12"},
 	 		"31", {\ar"22"},
 	 		"41", {\ar"32"},
 	 		"13", {\ar@<-.3ex>"04"},
 	 		"13", {\ar@<.3ex>"04"},
 	 		"23", {\ar"14"},
 	 		"33", {\ar"24"},
 	 		"43", {\ar"34"},
 	 		"40", {\ar"02"},
 	 		"42", {\ar"04"},
 	 	 \end{xy}
  	}
	\] With regards to arrows the following situations occur for $U=\langle i,j \rangle$ and $\x \in \L$: 	 
	\[
	\begin{xy} 0;<50pt,0pt>:<0pt,-18pt>:: 
		(0,-1.25) *+{R(\x)} ="01",
		(0,0) *+{R(\x+\vec{t}_{12})} ="02",
		(0,1.25) *+{R(\x+\vec{t}_{13})} ="03",
		(0,2.5) *+{R(\x+\vec{t}_{23})} ="04",
		(1.2,0.75) *+{U} ="1",
		(3.4,-1.25) *+{R(\x)} ="11",
		(3.4,0) *+{R(\x+\vec{t}_{12})} ="12",
		(3.4,1.25) *+{R(\x+\vec{t}_{13})} ="13",
		(3.4,2.5) *+{R(\x+\vec{t}_{23})} ="14",
		(2.2,0.75) *+{U} ="2", 
		"01", {\ar"1"},
		"02", {\ar"1"},
		"03", {\ar"1"},
		"04", {\ar"1"},
		"2", {\ar"11"},
		"2", {\ar"12"},
		"2", {\ar"13"},
		"2", {\ar"14"},
	\end{xy}
	\]
 	 \[
 	 \begin{xy} 0;<50pt,0pt>:<0pt,-20pt>:: 
 	 	(4,-1.25) *+{R(\x)} ="21",
 	 	(4,0) *+{R(\x+\vec{t}_{12})} ="22",
 	 	(4,1.25) *+{R(\x+\vec{t}_{13})} ="23",
 	 	(4,2.5) *+{R(\x+\vec{t}_{23})} ="24",
 	 	(5.7,-1.25) *+{R(\x+\x_4)} ="31",
 	 	(5.7,0) *+{R(\x+\vec{t}_{12}+\x_4)} ="32",
 	 	(5.7,1.25) *+{R(\x+\vec{t}_{13}+\x_4)} ="33",
 	 	(5.7,2.5) *+{R(\x+\vec{t}_{23}+\x_4)} ="34",
 	 	(7.3,-1.25) *+{R(\x)} ="41",
 	 	(7.3,0) *+{R(\x+\vec{t}_{12})} ="42",
 	 	(7.3,1.25) *+{R(\x+\vec{t}_{13})} ="43",
 	 	(7.3,2.5) *+{R(\x+\vec{t}_{23})} ="44",
 	 	(9.5,-1.25) *+{R(\x+\x_3)} ="51",
 	 	(9.5,0) *+{R(\x+\vec{t}_{12}+\x_3)} ="52",
 	 	(9.5,1.25) *+{R(\x+\vec{t}_{13}+\x_3)} ="53",
 	 	(9.5,2.5) *+{R(\x+\vec{t}_{23}+\x_3)} ="54",
 	 	"21", {\ar"31"},
 	 	"22", {\ar"32"},
 	 	"23", {\ar"33"},
 	 	"24", {\ar"34"},
 	 	"41", {\ar"51"},
 	 	"41", {\ar"53"},
 	 	"41", {\ar"54"},
 	 	"42", {\ar"52"},
 	 	"42", {\ar"53"},
 	 	"42", {\ar"54"},
 	 	"43", {\ar"51"},
 	 	"43", {\ar"52"},
 	 	"43", {\ar"53"},
 	 	"44", {\ar"51"},
 	 	"44", {\ar"52"},
 	 	"44", {\ar"54"},
 	 \end{xy}
 	 \]
 	  For simplicity, we denote them by
 	  \[
 	  \begin{xy} 0;<50pt,0pt>:<0pt,-20pt>:: 
 	  	(4,0) *+{(\x)} ="4",
 	  	(5,0) *+{U} ="5",
 	  	(6.5,0) *+{U} ="6",
 	  	(7.5,0) *+{(\x)} ="7",
 	  	"4", {\ar"5"},
 	  	"6", {\ar"7"},
 	  \end{xy}
 	  \]	  
 	 \[
 	 \begin{xy} 0;<50pt,0pt>:<0pt,-18pt>:: 
 	 	(4,0) *+{(\x)} ="4",
 	 	(5.3,0) *+{(\x+\x_4)} ="5",
 	 	(6.7,0) *+{(\x)} ="6",
 	 	(8,0) *+{(\x+\x_3).} ="7",
 	 	"4", {\ar"5"},
 	 	"6", {\ar@<-.2ex>"7"},
 	 	"6", {\ar@<.2ex>"7"},
 	 \end{xy}
 	 \]
 	In this example, one can check that the set $\mathcal{S}=\{\x+\vec{t}_{ab}\mid \x\in H,\ 1\le a<b\le 3 \}$, where $H:=\{\x_3,\x_3+\x_4,\x_4,2\x_4,3\x_4,\c\}$. 
 	 The $2$-tilting bundle $T$ in Proposition \ref{pre-prosition} is given by
 	$$T=U^{\rm CM} \oplus (\bigoplus_{\x\in \mathcal{S}}R(\x)).$$
 	%which is a both tilting object in $$
 	Take a decomposition $T=\langle 0,0 \rangle\oplus F$. 
 	Since $\Hom(F,\langle 0,0 \rangle)=0$, we may mutate and obtain  a $2$-tilting bundle 
 	$$M=\langle 1,0 \rangle \oplus \langle 2,0 \rangle \oplus \langle 2,2 \rangle \oplus (\bigoplus_{\x\in \mathcal{S}}R(\x))$$
 	on $\X$. However, $\langle 1,0 \rangle \oplus \langle 2,0 \rangle \oplus \langle 2,2 \rangle$ is not a tilting object in $\underline{\CM}^{\L} R$ since 
 		\begin{align*}
 		\underline{\Hom}(\langle 1,0 \rangle, \langle 2,2 \rangle[-1])
 		=\underline{\Hom}(\langle 1,0 \rangle, \langle 0,1 \rangle)=k.
 	\end{align*}
 	%This also gives a example, which is a slice in $\UU$ but not a tilting object in $\underline{\CM}^{\L} R$.
 	\end{example}
 	
 	In the rest, we present a non-example for the converse of  Theorem \ref{partical ans}.  
 	
 	\begin{example} \label{example 2} Let $(R,\L)$ be a GL hypersurface of type $(2,2,3,3)$, and $\X$ be the corresponding GL projective space. Keeping the above notations, we let $$M=(\bigoplus_{\vell\in [\s,\s+\vdelta)} U^{\vell} ) \oplus  (U^{\s+\vdelta}(-\w) ) \oplus(\bigoplus_{\x\in \mathcal{S}}R(\x)),$$
 		where $\vdelta=\x_3+\x_4$. 
 		Combining Theorem \ref{main theorem_1} and Proposition \ref{pro2}, $M$ is both a $2$-tilting bundle on $\X$ and a $2$-tilting object in $\underline{\CM}^{\L} R$. Take a decomposition $M=P\oplus W$, where $P=\bigoplus_{\x\in \mathcal{S}}R(\x)$. Now we claim that $\underline{\End}^{\L}_R(W)\neq \End^{\L}_R(W)$.  Notice that there exists a non-zero morphism $U^{\s+\x_3} \to R(\x_3)$. Moreover, there exists a sequence $R(\x_3) \to R(\x_3+\x_4) \to U^{\s+\vdelta}(-\w)$ of non-zero morphisms and thus the composition is injective by \cite[Proposition 3.28(c)]{HIMO}. Then the composition $$U^{\s+\x_3} \to R(\x_3) \to R(\x_3+\x_4) \to U^{\s+\vdelta}(-\w)$$
 		is also nonzero. But we have $\underline{\Hom}(U^{\s+\x_3},U^{\s+\vdelta}(-\w))=0$ by the diagram (a) in Proposition \ref{pro2}. Therefore we obtain the claim.  
 	\end{example}

 	\noindent {\bf Acknowledgements.} The authors would like to thank Shiquan Ruan for helpful discussions. Jianmin Chen and Weikang Weng were partially supported by the National Natural Science Foundation of China (Nos. 12371040 and 12131018).

			\vskip 5pt
	\noindent {\scriptsize   \noindent Jianmin Chen and Weikang Weng\\
		School of Mathematical Sciences, \\
		Xiamen University, Xiamen, 361005, Fujian, PR China.\\
		E-mails: chenjianmin@xmu.edu.cn, wkweng@stu.xmu.edu.cn\\ }
	\vskip 3pt
	
\end{document}